\documentclass[a4paper]{amsart}

\usepackage{amsmath, amssymb, amsthm}
\usepackage{graphicx} 
\usepackage{url}
\usepackage{hyperref}
\usepackage{float}

\theoremstyle{plain}
\newtheorem{introdefinition}{Definition}
\newtheorem{introtheorem}{Theorem}

\theoremstyle{plain}
\newtheorem{theorem}{Theorem}[section]
\newtheorem{proposition}[theorem]{Proposition}
\newtheorem{lemma}[theorem]{Lemma}

\theoremstyle{definition}
\newtheorem{definition}[theorem]{Definition}

\newtheorem{example}[theorem]{Example}

\newtheorem{coro}[theorem]{Corollary}

\newcommand{\C}{\mathbb{C}}
\renewcommand{\ge}{\geqslant}

\renewcommand{\le}{\leqslant}
\newcommand{\N}{\mathbb{N}}
\newcommand{\R}{\mathbb{R}}
\newcommand{\Sph}{\mathbb{S}}

\newcommand{\curl}{\mathrm{curl}}

\newcommand{\En}{\mathrm{E}}

\newcommand{\Flux}[3]{\mathrm{Flux}(#1,#2, #3)}

\newcommand{\Haar}{\mathrm{Haar}}
\newcommand{\Hel}{\mathrm{Hel}}

\newcommand{\kn}{\mathcal{K}}

\newcommand{\Lk}{\mathrm{Lk}}

\newcommand{\tr}[3]{\mathrm{tks}_{#3}(#1,#2)}
\newcommand{\Tr}[2]{\mathrm{Tks}(#1,#2)}
\newcommand{\ktr}[2]{\mathrm{tk}_{#2}(#1)}
\newcommand{\kTr}[1]{\mathrm{Tk}(#1)}

\begin{document}

\title{The trunkenness of a volume-preserving vector field}
\author{Pierre Dehornoy, Ana Rechtman}
\AtEndDocument{\bigskip{\footnotesize%
  \textsc{Univ. Grenoble Alpes, CNRS, IF, F-38000 Grenoble, France} \par  
  \url{pierre.dehornoy@univ-grenoble-alpes.fr} \par
  \url{http://www-fourier.ujf-grenoble.fr/~dehornop/}\par
}}
\AtEndDocument{\bigskip{\footnotesize%
  \textsc{Instituto de Matem\'aticas, Universidad Nacional Aut—noma de M\'exico, Ciudad Universitaria, 04510 Ciudad de M\'exico, M\'exico} \par  
  \url{rechtman@im.unam.mx} \par
  \url{http://www.matem.unam.mx/fsd/rechtman} \par 
}}
\date{first version Mai 7th, 2016; last correction July 18th, 2017}

\begin{abstract}
We construct a new invariant---the \emph{trunkenness}---for volume-perserving vector fields on~$\Sph^3$ up to volume-preserving diffeomorphism. 
We prove that the trunkenness is independent from the helicity and that it is the limit of a knot invariant (called the trunk) computed on long pieces of orbits.
\end{abstract}


\maketitle

\tableofcontents


The problem we address here is the construction of new invariants of volume-preserving vector fields on~$\Sph^3$ or on compact domains of~$\R^3$ up to volume-preserving diffeomorphisms. 
This problem is motivated by at least two physical situations. 
First if $v$ is the velocity field of a time-dependant ideal fluid satisfying the Euler equations (ideal hydrodynamics) then its vorticity field $\curl\,v$ is transported by the flow of~$v$~\cite{Helmholtz}.
Second if $B$ is the magnetic field in an incompressible plasma (ideal magnetodynamics), then $B$ turns out to be transported by the velocity field as long as the latter does not develop singularities~\cite{Woltjer}. 
In these contexts, invariants of $\curl\,v$ or $B$ up to volume-preserving diffeomorphisms yield time-independent invariants of the system. 

Not so many such invariants exist.
The first one was discovered by William Thomson~\cite{Thomson}: if the considered field has a periodic orbit or a periodic tube, then its knot type is an invariant (this remark led to the development of knot theory by Peter G. Tait~\cite{Tait}). 
However it may not be easy to find periodic orbits, and even then such an invariant only takes a small part of the field into account. 

The main known invariant is called \emph{helicity}. 
It is defined by the formula~$\Hel(v) = \int v\cdot u$, where $u=\curl^{-1}(v)$ is an arbitrary vector-potential of~$v$. 
It was discovered by Woltjer, Moreau, and Moffatt~\cite{Woltjer, Moreau, Moffatt}. 
Helicity is easy to compute or to approximate since it is enough to exhibit a vector-potential of the considered vector field, to take the scalar product and to integrate.
The connection with knot theory was sketched by Moffatt~\cite{Moffatt} and deepened by Arnold~\cite{Arnold} as follows. 
Denote by~$k_X(p,t)$ a loop starting at the point~$p$, tangent to the vector field~$X$ for a time~$t$ and closed by an arbitrary segment of bounded length. 
Denote by~$\Lk$ the linking number of loops. 
Arnold showed that for almost every~$p_1, p_2$, the limit~$\displaystyle{\lim_{t_1,t_2\to\infty} \frac1{t_1t_2}\Lk(k_X(p_1,t_1),k_X(p_2,t_2))}$ exists (see also~\cite{Vogel} for a corrected statement). 
Moreover if~$X$ is ergodic the limit coincides almost everywhere with~$\Hel(X)$ (for a non-ergodic vector field, one has to average the previous limit).

The idea of considering knot invariants of long pieces of orbits of the vector field was pursued by Gambaudo and Ghys~\cite{GG} who considered $\omega$-signatures of knots, Baader~\cite{Baader} who considered linear saddle invariants, and Baader and March\'e~\cite{BM} who considered Vassiliev's finite type invariants. 
In every case, it is shown that $\lim_{t\to\infty} \frac{1}{t^n}V(k_X(p,t))$ exists, where $V$ is the considered invariant and $n$ a suitable exponent called the \emph{order} of the asymptotic invariant. 
However all these constructions have the drawback that they do not yield any new invariant for ergodic vector fields, as in this case the obtained limits are all functions of the helicity. 

Recently, it was proved by Kudryavtseva \cite{Kud} for vector fields obtained by suspending an area-preserving diffeomorphism of a surface and for non-vanishing vector fields, and then by Enciso, Peralta-Salas and Torres de Lizaur \cite{EPT} for arbitrary volume-preserving vector fields, that every invariant that is \emph{regular integral} (in the sense that its Fr\'echet derivative is the integral of a continuous kernel) is a function of helicity, see the cited articles for precise statements.
These results give a satisfactory explanation of why most constructions yield invariants that are functions of helicity for ergodic vector fields. 
However they do not rule out the existence of other invariants, but imply that such invariants cannot be \emph{too} regular. 

An example of such another invariant is the asymptotic crossing number considered by Freedman and He~\cite{FH}. 
The advantage is that it is not proportional to helicity, but the disadvantage is that it is hard to compute, even on simple examples. 

\bigskip 
In this article we consider a less known knot invariant called the trunk (see Definition~\ref{D:TrunkKnot} below). 
It was defined by Ozawa~\cite{Ozawa}, building on the concept of \emph{thin position} that was introduced by Gabai~\cite{Gabai} for solving the R-conjecture. 
Less famous that the invariants previously studied in the context of vector fields, the trunk has the advantage that its definition relies on surfaces transverse to the considered knot, so that it is easy to transcript in the context of vector fields. 
The invariant depends on an invariant measure for the flow of the vector field that may or may not be a volume, and is invariant under diffeomorphisms that preserve this measure.
Given a $\mu$-preserving vector field~$X$ and a surface~$S$, the \emph{geometric flux} through~$S$ is the infinitesimal volume that crosses~$S$ in both directions (see Definition~\ref{D:Flux}), it is denoted by~$\Flux X \mu S$. 
Our invariant is a minimax of the geometric flux, where one minimizes over all height functions and maximizes over the levels of the considered height function.

\begin{introdefinition}\label{def-intro}
	Assume that $X$ is a vector field on~$\Sph^3$ or on a compact domain of~$\R^3$ that preserves a probability measure~$\mu$. Denote by~$\phi_X$ the flow of~$X$. 
	The \emph{trunkenness} of~$X$ with respect to~$\mu$ is
	\[ \Tr X \mu := \inf_{\substack{h\mathrm{~height}\\\mathrm{function}}} \ \max_{t\in[0,1]} \Flux{X}\mu{h^{-1}(t)}\ = \inf_{\substack{h\mathrm{~height}\\\mathrm{function}}} \ \max_{t\in[0,1]}\ \lim_{\epsilon\to 0}\ \frac 1 \epsilon \mu(\phi_X^{[0,\epsilon]}(h^{-1}(t))). \]
\end{introdefinition}

By a height function on $\Sph^3$ we refer to a function with only two
singular points and whose level sets are 2-dimensional spheres. On
$\R^3$, the level sets of a height function all are topological planes.

From the definition, it is straightforward that the trunkenness of a vector
field is invariant under diffeomorphisms that preserve the measure
$\mu$. More is true, the trunkenness is invariant under homeomorphisms that
preserve $\mu$.

\begin{introtheorem}\label{T:invariance}
	Assume that $X_1$ and $X_2$ are vector fields on ~$\Sph^3$ or on a compact domain of~$\R^3$ that preserve a probability measure~$\mu$ and that there is a $\mu$-preserving homeomorphism $f$ that conjugates the flows of $X_1$ and $X_2$. 
	Then we have
	\[\Tr {X_1}\mu=\Tr{X_2}\mu.\]
\end{introtheorem}

What we do in this paper is to prove several properties of this new invariant. 
The first one is a continuity result that implies that the trunkenness of a vector field is an asymptotic invariant of order~$1$. 
For $K$ a knot, we denote by $\kTr K$ its trunk (see Definition~\ref{D:TrunkKnot} below).
 
\begin{introtheorem}
	\label{T:Continuity}	
	Suppose that $(X_n, \mu_n)_{n\in\N}$ is a sequence of measure-preserving vector fields such that $(X_n)_{n\in\N}$ converges to~$X$ and $(\mu_n)_{n\in\N}$ converges to~$\mu$ in the weak-$*$ sense. 
	Then we have
	\[\lim_{n\to\infty} \Tr{X_n} {\mu_n} = \Tr X \mu.\] 	
	In particular if $X$ is ergodic with respect to~$\mu$ then,
        for $\mu$-almost every~$p$, the limit 
        \[ \lim_{t\to\infty} \frac 1 t \kTr{k_X(p,t)} \]
	exists and is equal to~$\Tr X \mu$.
\end{introtheorem}

This continuity result then allows us to compute the trunkenness of some explicit vector fields on~$\Sph^3$ called Seifert flows. 
These computations in turn show that the trunkenness is not dictated by helicity, even in the case of ergodic vector fields, thus contrasting with most previously known knot-theoretical constructions. 

\begin{introtheorem}
	\label{T:Seifert}
	There is no function~$f$ such that for every ergodic volume-preserving vector field~$X$ on~$\Sph^3$ one has $\Tr X \mu = f(\Hel(X, \mu))$.
\end{introtheorem}

Finally we adress the question of what happens if for a non-singular vector field on $\Sph^3$ there is a function that achieves the trunkenness, or in other words if the infimum in Definition~\ref{def-intro} is a minimum. 

\begin{introtheorem}\label{T:minperiodic} 
Let $X$ be a  non-singular vector field on $\Sph^3$ preserving the measure $\mu$ and $h$ a height function  such that 
\[
	\Tr X \mu= \max_{t\in[0,1]} \Flux X \mu {h^{-1}(t)}
\]
Then $X$ has an unknotted periodic orbit.
\end{introtheorem}

One of the main motivations for constructing topological invariants of a vector field~$X$ is to find lower bound on the energies~$\En_p(X) := \int \vert X\vert^p\,d\mu$. 
Indeed since a topological invariant yields a time-independent invariant of the physical system, an energy bound in term of a topological invariant will also be time-independent, although the energy may vary when the vector field is transported under (volume-preserving) diffeomorphisms. 
Such energy bounds exist for the helicity and for the asymptotic crossing number.
We do not know whether the trunkenness bounds the energy.

The plan of the article is as follows.
First we recall in Section~\ref{S:TrunkKnot} the definition of the trunk of a knot in order to make the definition for vector fields natural. 
Then we define the trunkenness of a measure-preserving vector field and
prove Theorems~\ref{T:invariance} and \ref{T:Continuity} in Section~\ref{S:TrunkField}.
Using this result we compute of trunkenness of Seifert vector fields and
prove Theorem~\ref{T:Seifert} in Section~\ref{S:Seifert}. 
We prove Theorem~\ref{T:minperiodic} in Section~\ref{sec-periodic}. 
Finally in Section~\ref{S:Examples} we compute the trunkenness of some vector fields
supported in the tubular neighborhood of a link. 

\vspace{2mm}
\noindent\emph{Acknowledgements.} We thank Michel Boileau who suggested to study the trunk instead of the genus for vector fields during a visit of P.D. to Toulouse in 2013.


\section{Trunk of knots}
\label{S:TrunkKnot}

For $K$ a knot, we denote by~$\kn$ the set of all embeddings of~$K$ into~$\R^3$. 
The \emph{standard height function} on~$\R^3$ is the function~$h_z:\R^3\to\R, (x,y,z)\mapsto z$. 
Every level $h_z^{-1}(t)$ is a 2-dimensional plane.
An embedding~$k\in\kn$ is said to be in \emph{Morse position} with respect to~$h_z$ if the restriction of~$h_z$ to~$k$ is a Morse function. In this case there are only finitely many points at which $k$ is tangent to a level of~$h_z$.

\begin{definition}
	\label{D:TrunkKnot}
	Assume that $k$ is an embedded knot in $\R^3$ that is in Morse position with respect to~$h_z$. The \emph{trunk of the curve~$k$ relatively to~$h_z$} is 
	\[\ktr k {h_z} := \max_{t\in\R}\sharp \{ k\cap h_z^{-1}(t)\}.\] 
	
	The \emph{trunk of a knot~$K$} is then defined by
	\[ \kTr K := \min_{k\in\kn}\, \ktr k {h_z} = \min_{k\in\kn}\,\max_{t\in\R}\,\sharp \{ k\cap h_z^{-1}(t)\}. \]
\end{definition}

The trunk of a knot was defined by M. Ozawa~\cite{Ozawa} and
motivated by D. Gabai's definition of the waist of a knot \cite{Gabai}.

\begin{example}
A knot is trivial if and only if its trunk equals~$2$. Indeed the embedding as the boundary of a vertical disc shows that the trunk is less than or equal to $2$, and  every embedding in Morse position of the trivial knot has to intersect some horizontal plane in at least two points. 
Conversely, if the trunk of a knot is equal to 2, then it admits an embedding that intersects every horizontal plane in at most two points. 
The union of the segments that connect these pairs of points is a disc bounded by the knot, implying that the knot is trivial.
\end{example}

\begin{figure}
	\includegraphics[width=.4\textwidth]{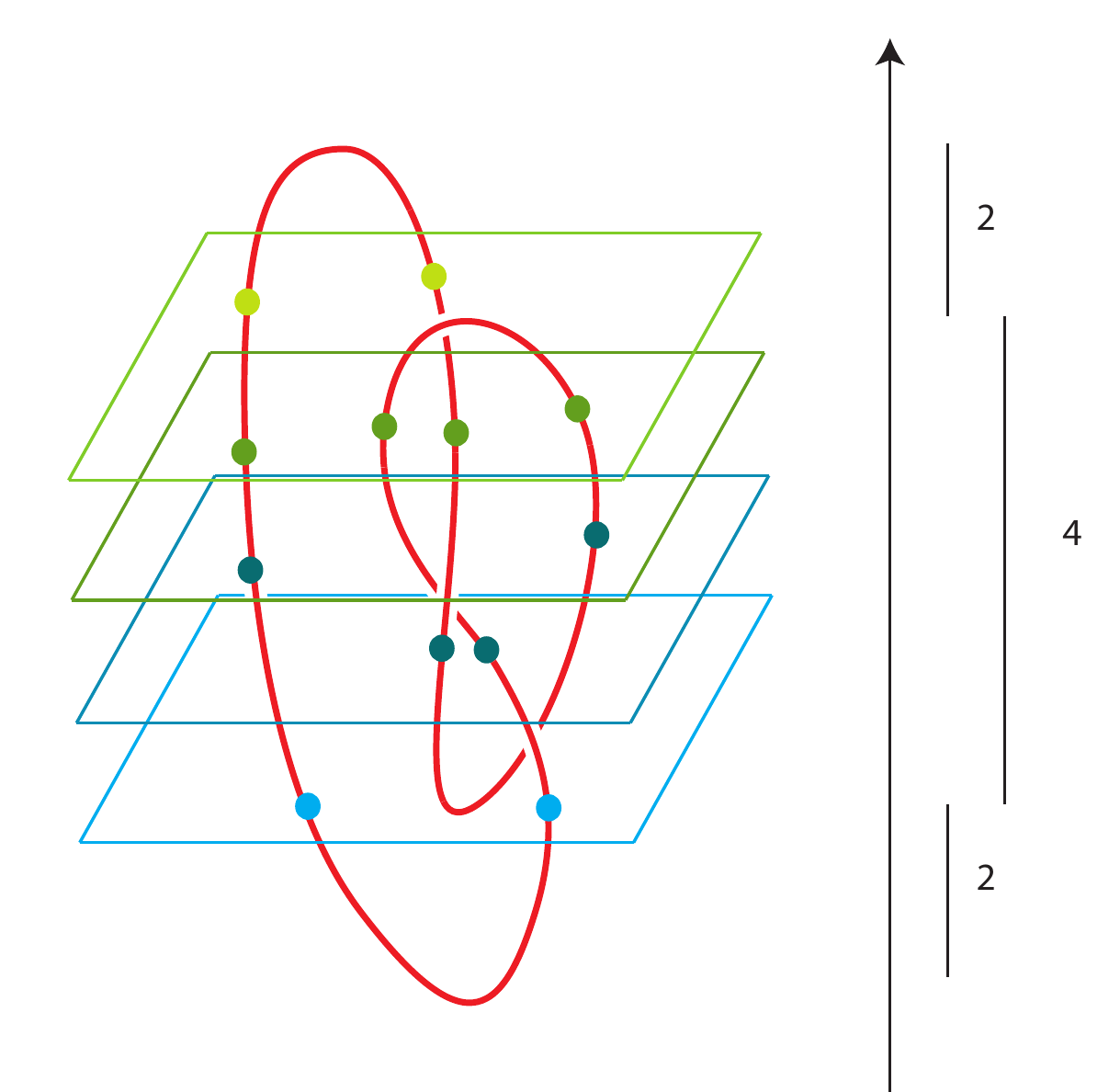}
	\caption{\small The trunk of the trefoil knot: the maximal number of intersection points between a horizontal level and the proposed embedding is 4. This number cannot be reduced under isotopy, hence the trunk of the trefoil is 4.}
	\label{F:TrunkKnot}
\end{figure}

\begin{example}
	For $p,q$ in~$\N$, the torus knot~$T(p,q)$ can be realized as the closure of a braid with $q$ strands, yielding $\kTr {T(p,q)}\le 2q$. 
	By symmetry one also gets~$\kTr {T(p,q)}\le 2p$. 
	Actually, one can prove~$\kTr{T(p,q)}=2\min(p,q)$, see Remark~1.2 in \cite{Ozawa}.
\end{example}

Instead of fixing the function and changing the knot up to isotopy, one can fix the knot and change the function up to orientation preserving diffeomorphism.
With this in mind, one defines a \emph{height function} on~$\R^3$ as a function obtained by precomposing~$h_z$ by a diffeomorphism, that is, a function of the form 
\begin{eqnarray*}
	h:\R^3 & \to &\R\\ 
	(x,y,z) & \mapsto & h_z(\phi(x,y,z))
\end{eqnarray*} 
for $\phi$ an orientation preserving diffeomorphism of $\R^3$. 
In particular, a height function is a function whose levels are smooth planes. 
For $K$ a knot and $k$ a fixed embedding of~$K$ in~$\R^3$, one can then define 
	\[\ktr k h := \max_{t\in\R}\sharp \{ k\cap h^{-1}(t)\},\] 
so that we have the alternative definition
	\begin{equation}\label{eq:defTr}
	\kTr K = \min_{\substack{h\mathrm{~height}\\\mathrm{function}}} \,\ktr k h =\min_{\substack{h\mathrm{~height}\\\mathrm{function}}}\,\max_{t\in\R}\,\sharp \{ k\cap h^{-1}(t)\}. 
	\end{equation}


\section{Trunkenness of measure preserving vector fields}
\label{S:TrunkField}

We use the definition of Equation~\eqref{eq:defTr} to define the trunkenness of a vector field with respect to an invariant measure. 
The main question then concerns the analog of the number of intersection points of a surface with a curve when the curve is replaced by a vector field. 
A natural answer is the \emph{geometric flux}. 
If $X$ is a vector field that preserves a measure~$\mu$ given by a volume element~$\Omega$, one can then consider the 2-form~$\iota_X\Omega$. 
For $S$ a piece of oriented surface that is positively transverse to~$X$, the integral $\int_S\iota_X\Omega$ computes the instantaneous volume that crosses~$S$. 
In other words, by Fubini Theorem we have $\mu(\phi^{[0,t]}(S)) = (\int_S\iota_X\Omega)\cdot t$. 
On the other hand if $S$ is negatively transverse to~$X$ we have $\mu(\phi^{[0,t]}(S)) = -(\int_S\iota_X\Omega) \cdot t$. 
Therefore in this case, for any surface~$S$, the instantaneous volume crossing~$S$ is given by $\int_S \vert\iota_X\Omega\vert$. 
Now if the measure~$\mu$ is not given by integrating a volume form one cannot consider the above integral, but the quantity $\mu(\phi^{[0,t]}(S))$ still makes sense for any piece of surface~$S$ (see Figure~\ref{F:Surface}).

\begin{figure}
	\includegraphics[width=.4\textwidth]{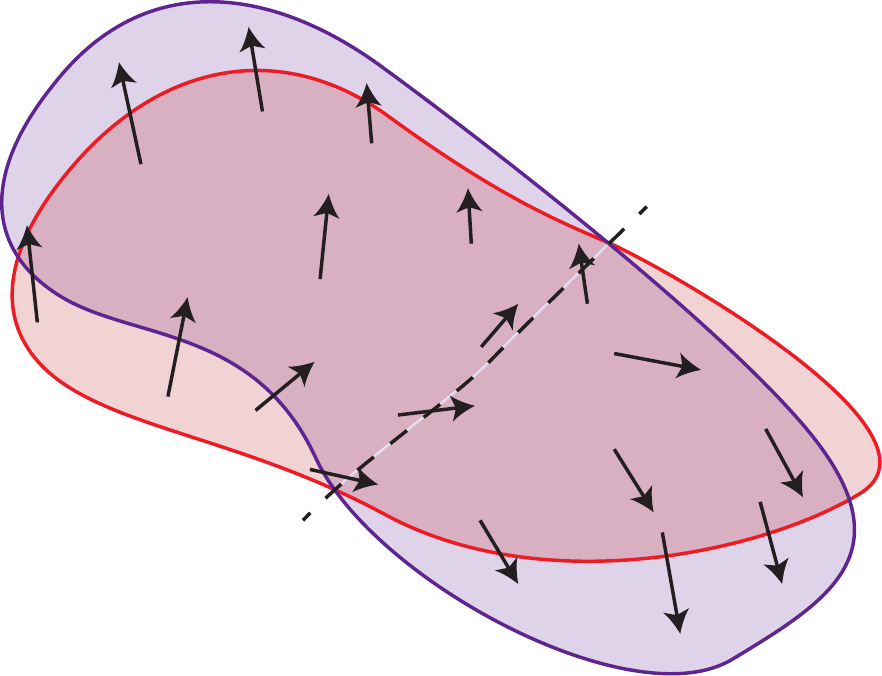}
	\caption{\small For $S$ a piece of a surface (in red), the domain $\phi^{[0,\epsilon]}(S)$ is roughly the part of the space located between $S$ and $\phi^\epsilon(S)$. 
	It is pinched around those points where the vector field is tangent to~$S$. 
	Its volume is~$\epsilon\cdot\Flux X \mu S$ at the first order.}
	\label{F:Surface}
\end{figure}

\begin{definition}\label{D:Flux}
	For $X$ a vector field that preserves a measure~$\mu$ and for $S$ a surface, the \emph{geometric flux} of~$(X,\mu)$ through~$S$ is
	\[\Flux X \mu S :=  \lim_{\epsilon\to 0}\frac{1}{\epsilon} \,\mu(\phi^{[0,\epsilon]}(S)). \]
\end{definition}

This definition generalizes the number of intersection points of a knot with a surface. 
Indeed one can see an embedding~$k$ of a knot $K$ as a vector field with a particular invariant measure in the following way: 
consider a non-singular vector field~$X_k$ that is tangent to~$k$ at every point and denote by~$\phi_k^t$ the induced flow. Since $k$ is closed, $\phi_k$ is $T_k$-periodic for some~$T_k>0$. The Dirac linear measure associated to~$X_k$ is defined by
$$\mu_k(A) = \mathrm{Leb}(\{t\in[0,T_k], \phi_k^t(x)\in A\})$$ 
where $A$ is a measurable set and $x$ an arbitrary point on $k$.
The measure~$\mu_k$ is $X_k$-invariant and has total mass~$T_k$.
In this setting, for $S$ a surface that intersects~$k$ in finitely many points, a point $p$ in the set $k\cap S$ has $\mu_k$-measure zero and thus cannot be detected by the measure. 
But by definition of~$\mu_k$ the set $\phi_k^{[0,\epsilon]}(p)$ is an arc of~$k$ of $\mu_k$-measure~$\epsilon$ and since $k\cap S$ is made of finitely many points, for $\epsilon$ small enough, the set $\phi_k^{[0,\epsilon]}(k\cap S)$ has $\mu_k$-measure exactly $\epsilon\cdot \sharp\{ k\cap S\}$.
In other words, one has 
	\[\sharp\{ k\cap S\} = \lim_{\epsilon\to 0}\frac{1}{\epsilon} \,\mu_k(\phi_k^{[0,\epsilon]}(k\cap S)).\]
As $\mu_k$ is concentrated on~$k$, we thus have 
\begin{equation}\label{eq:measureintersection}
	\sharp \{ k\cap S\} = \lim_{\epsilon\to 0}\frac{1}{\epsilon} \,\mu_k(\phi_k^{[0,\epsilon]}(S)) = \Flux {X_k} {\mu_k} S,
\end{equation}
so the geometric flux indeed generalizes the number of intersection points.

\medskip
We now mimic for vector fields the definition of the trunk of a knot.
In order to have a well-defined maximum, in what follows we assume the vector fields are on  a compact domain~$D^3\subset\R^3$ or on the 3-sphere~$\Sph^3=\R^3\cup\infty$. 
In the later case, we define the \emph{standard height function}
\begin{eqnarray*}
	h_0:\Sph^3 & \to & [0,1]\\ 
	(x,y,z) & \mapsto & 1-\frac{1} {1+x^2+y^2+z^2}.
\end{eqnarray*} 
The levels~$h_0^{-1}(0)$ and~$h_0^{-1}(1)$ consist of the points~$(0,0,0)$ and~$\infty$ respectively and every other level $h_0^{-1}(t)$ is a 2-dimensional sphere of radius~$\sqrt{t/(1-t)}$ centered at the origin.
A \emph{height function} on~$\Sph^3$ is then a function obtained by precomposing by an orientation preserving  diffeomorphism $\phi$ of $\Sph^3$, that is, a function of the form $h:\Sph^3\to[0,1], (x,y,z)\mapsto h_0(\phi(x,y,z))$.

\begin{definition}
	\label{D:TrunkField}
	Let $X$ be a vector field whose flow  preserves a measure~$\mu$ on a compact domain of $\R^3$ or on $\Sph^3$, and $h$  a height function. We set
	\[\tr X \mu h := \max_{t\in[0,1]}\ \Flux X \mu {h^{-1}(t)} = \max_{t\in[0,1]}\,\lim_{\epsilon\to 0}\,\frac 1 \epsilon \,\mu(\phi^{[0,\epsilon]}(h^{-1}(t))).\] 
	The \emph{trunkenness of~$(X, \mu)$} is defined as
	\begin{eqnarray*}
		\Tr X \mu := \inf_{\substack{h\mathrm{~height}\\\mathrm{function}}}\, \tr X \mu h 
		&=& \inf_{\substack{h\mathrm{~height}\\\mathrm{function}}} \,\max_{t\in[0,1]}\, \Flux X \mu {h^{-1}(t)}\\
		&=&\inf_{\substack{h\mathrm{~height}\\\mathrm{function}}} \,\max_{t\in[0,1]}\,\lim_{\epsilon\to 0}\,\frac 1 \epsilon\, \mu(\phi^{[0,\epsilon]}(h^{-1}(t))). 
	\end{eqnarray*}
\end{definition}

Note that we can only consider an infimum instead of a minimum as in the case of knots. 
In Section~\ref{sec-periodic} we prove that for non-singular vector fields, if the trunkenness is a minimum then the vector field possesses an unknotted periodic orbit. 

If the invariant measure~$\mu$ is given by the integration of a volume form~$\Omega$, we get the alternative definitions
	\[\tr X \Omega h = \max_{t\in[0,1]} \, \int_{h^{-1}(t)} \vert\iota_X\Omega\vert, \quad\mathrm{and}\quad \Tr X \Omega = \inf_{\substack{h\mathrm{~height}\\\mathrm{function}}} \,\max_{t\in[0,1]} \, \int_{h^{-1}(t)} \vert\iota_X\Omega\vert.\] 
	
\medskip

We now prove the invariance under homeomorphism.
	
\begin{proof}[Proof of Theorem~\ref{T:invariance}]
Assume without loss of generality that there is a homeomorphism~$f$ such that $f\cdot \phi_{X_1}^t=\phi_{X_2}^t\cdot f$, with $\phi_{X_i}^t$ the flow of $X_i$, and assume by contradiction that $\Tr {X_1} \mu<\Tr {X_2} \mu$. 
Let $0<\delta=\Tr {X_2}\mu-\Tr{X_1}\mu$. 
Take $h_n$ to be a sequence of (differentiable) height functions such that
$$\lim_{n\to \infty}\tr{X_1}\mu {h_n}=\Tr {X_1}\mu.$$

Consider the continuous functions $\widetilde{h}_n=h_n\cdot f^{-1}$ and for each $n$ consider a differentiable function $H_n$ that is arbitrarily $C^0$-close to $\widetilde{h}_n$. 
Hence, for each $n$, the level sets of $H_n$ and $\widetilde{h}_n$ are arbitrarily close. 
Observe that we can assume that $H_n$ is a height function (by taking $H_n$ such that its level sets are all spheres or all planes, depending on the domain of the vector field). 
Thus we can choose $H_n$ so that, for every $t\in[0,1]$ and for $\epsilon$ small enough, we have
\[
	|\mu(\phi_{X_2}^{[0,\epsilon]}(H_n^{-1}(t)))-\mu(\phi_{X_2}^{[0,\epsilon]}(\widetilde{h}_n^{-1}(t)))|<\frac{\delta}{4}.
\]
Fix $n$ large enough so that
$\tr {X_1}\mu{h_n}-\Tr{X_1}\mu<\frac{\delta}{4}$, and for such $n$ set $T$ in $[0,1]$ so that 
\[\tr{X_2}\mu{H_n}=\Flux {X_2} \mu {H_n^{-1}(T)}.
\] 
We have two possibilities:
\begin{enumerate}
	\item 
	If $\Flux {X_2} \mu {H_n^{-1}(T)}\geq \Flux {X_2} \mu {\widetilde{h}_n^{-1}(T)}$, then for $\epsilon$ small enough
	\[\mu(\phi_{X_2}^{[0,\epsilon]}(H_n^{-1}(T)))\leq \frac{\delta}{4}+\mu(\phi_{X_2}^{[0,\epsilon]}(\widetilde{h}_n^{-1}(T))).\]
	Hence
	\begin{eqnarray*}
		\Tr{X_2}\mu & \leq & \tr{X_2}\mu{H_n}  
			=  \lim_{\epsilon\to 0}\frac{1}{\epsilon}\mu(\phi_{X_2}^{[0,\epsilon]}(H_n^{-1}(T)))\\
		& \leq & \frac{\delta}{4}+\lim_{\epsilon\to  0}\frac{1}{\epsilon}\mu(\phi_{X_2}^{[0,\epsilon]}(\widetilde{h}_n^{-1}(T)))\\
		& = & \frac{\delta}{4}+\lim_{\epsilon\to 0}\frac{1}{\epsilon}\mu(\phi_{X_1}^{[0,\epsilon]}(h_n^{-1}(T)))\\
		& \leq & \frac{\delta}{4} +\tr{X_1}\mu{h_n} <  \frac{\delta}{2}+\Tr{X_1}\mu.
	\end{eqnarray*}
	Thus, $\delta=\Tr{X_2}\mu-\Tr{X_1}\mu< \frac{\delta}{2}$, 
	a contradiction proving Theorem~\ref{T:invariance} in this case.

	\item If $\Flux {X_2} \mu {H_n^{-1}(T)}< \Flux {X_2} \mu {\widetilde{h}_n^{-1}(T)}$, then
	\begin{eqnarray*}
		\tr{X_2}\mu {H_n} 
		& = & \Flux {X_2} \mu {H_n^{-1}(T)}\\
		& < & \Flux {X_2} \mu {\widetilde{h}_n^{-1}(T)} = \Flux {X_1} \mu {h_n^{-1}(T)}\\
		& \leq & \tr{X_1}\mu {h_n} <  \Tr {X_1} \mu+\frac{\delta}{4}<\Tr {X_2} \mu.
	\end{eqnarray*}
	Thus, $\tr{X_2}\mu {H_n}<\Tr {X_2} \mu$, a contradiction to the definition of the trunkenness. 
\end{enumerate}
This finishes the proof of Theorem~\ref{T:invariance}.
\end{proof}

Computing the trunkenness of a vector field is not easy in general. 
Considering a given height function gives an upper bound on the trunkenness, but lower bounds are harder to find. 
Theorem~\ref{T:Continuity} is a continuity result that provides a useful tool.

	
\begin{proof}[Proof of Theorem~\ref{T:Continuity}]
We begin with the first part of the theorem, namely we prove that if $(X_n, \mu_n)_{n\in\N}$ is a sequence of measure-preserving vector fields such that $(X_n)_{n\in\N}$ converges to~$X$ and $(\mu_n)_{n\in\N}$ converges to~$\mu$ in the weak-$*$ sense, then we have $\lim_{n\to\infty} \Tr{X_n} {\mu_n} = \Tr X \mu.$

Fix $\epsilon>0$. The convergence hypothesis implies that for every
surface~$S$, for every $\delta>0$ small and for $n$ big enough, we have 
\begin{equation}\label{eq:convergence}
|\mu_n(\phi_{X_n}^{[0,\delta]}(S)) -\mu(\phi_{X}^{[0,\delta]}(S))|\leq \epsilon, 
\end{equation}
	where $\phi_{X}^t$ and $\phi_{X_n}^t$ denote the flows of the vector fields $X$ and $X_n$, respectively.
	
	Assume that $\Tr{X_n}{\mu_n}$ does not converge to $\Tr X \mu$, then for any $N\in \mathbb{N}$ there exists $n\geq N$ such that
	$$|\Tr{X_n}{\mu_n}-\Tr X \mu|>3\epsilon.$$
	Fix $n$ big enough satisfying the last inequality and \eqref{eq:convergence}. Then either 
	$$\Tr{X_n}{\mu_n}<\Tr X \mu \qquad \text{or} \qquad  \Tr X \mu<\Tr{X_n}{\mu_n}.$$ 
	Next we analyse these two cases and deduce a contradiction for each.
	
	If $\Tr X \mu<\Tr{X_n}{\mu_n}$, consider a sequence of height functions $h_k$ such that 
	$$\lim_{k\to \infty} \tr{X}\mu {h_k}=\Tr X \mu.$$
Modulo extracting a subsequence we can assume that for all $k$ we have that $0\leq \tr X\mu{h_k}-\Tr X \mu\leq \epsilon$. Observe that this difference is always positive by the definition of the trunkenness. Then
	\begin{eqnarray*}
	3\epsilon & < & \Tr{X_n}{\mu_n}-\Tr X \mu \\
	& = & \Tr{X_n}{\mu_n}-\tr X\mu{h_k}+\tr X\mu{h_k} - \Tr X \mu,
	\end{eqnarray*}
	hence,
	\begin{eqnarray*}
	2\epsilon & < & \Tr{X_n}{\mu_n}-\tr X\mu{h_k} +\tr{X_n}{\mu_n}{h_k} -\tr{X_n}{\mu_n}{h_k}\\
	& < & \Tr{X_n}{\mu_n}-\tr{X_n}{\mu_n}{h_k} +|\tr{X_n}{\mu_n}{h_k}-\tr X\mu{h_k}|\\
	\end{eqnarray*}
	For $k$ fixed, since $\mu_n$ tends to $\mu$ in the weak-$*$ sense, we can assume that the term\linebreak $|\tr{X_n}{\mu_n}{h_k}-\tr X\mu{h_k}|$ is smaller than $\epsilon$ by possibly taking $n$ larger. 
	Then we get 
	$2\epsilon <  \Tr{X_n}{\mu_n}-\tr{X_n}{\mu_n}{h_k}  +\epsilon.$
	We deduce $\epsilon+\tr{X_n}{\mu_n}{h_k}<\Tr{X_n}{\mu_n}$, a contradiction to the definition since we have
	\[\epsilon+\tr{X_n}{\mu_n}{h_k}<\Tr{X_n}{\mu_n}\leq \tr{X_n}{\mu_n}{h_k}.\]
	
	\bigskip
	
	The other case is similar. 
	Assume now $\Tr{X_n}{\mu_n}<\Tr X \mu$ and for each $n$ consider a sequence of height functions $h_{n,k}$ such that 
	\[\lim_{k\to \infty} \tr{X_n}{\mu_n}{h_{n,k}}=\Tr{X_n}{\mu_n}.\] 
	As in the previous case, we assume that for all $k$ we have
        $0\leq \tr{X_n}{\mu_n}{h_{n,k}}-\Tr{X_n}{\mu_n}\leq \epsilon$. 
        Then, for $n$ large enough,
	\begin{eqnarray*}
	3\epsilon & < & \Tr X \mu-\Tr{X_n}{\mu_n} \\
	& = & \Tr X \mu-\tr{X_n}{\mu_n}{h_{n,k}}+\tr{X_n}{\mu_n}{h_{n,k}} - \tr X \mu {h_{n,k}}\\
	& & +\tr X \mu {h_{n,k}}- \Tr{X_n}{\mu_n}\\
	& \leq & \Tr X \mu-\tr{X}\mu{h_{n,k}}+\epsilon + |\tr X \mu {h_{n,k}}-\tr{X_n}{\mu_n}{h_{n,k}}|\\
	& \leq & \Tr X \mu-\tr X \mu{h_{n,k}}	+2\epsilon.
	\end{eqnarray*}
	We conclude that $\epsilon+\tr X \mu {h_{n,k}}<\Tr X \mu$, a contradiction to the definition of $\Tr X \mu$.
	Thus, for $n$ large enough, $|\Tr{X_n}{\mu_n}-\Tr X \mu|<3\epsilon$. 
	This proves the first part of Theorem~\ref{T:Continuity}.
	
	\bigskip

The second part of Theorem~\ref{T:Continuity} gives an asymptotic interpretation to the trunkenness of a vector field. 
	Consider a sequence of knots $(K_n)_{n\in\N}$ that support vector fields~$X_n$ and  denote by~$(t_n)_{n\in\N}$ the respective periods. From Equation \eqref{eq:measureintersection} we get that for each $n$
	$$\sharp\{ K_n\cap h^{-1}(t)\} = \lim_{\epsilon\to 0}\frac{1}{\epsilon} \,\mu_n(\phi_n^{[0,\epsilon]}(h^{-1}(t))),$$
	where $\mu_n$ are the mesures supported by the knots and
        $\phi_n^t$ is the flow of $X_n$.
	 
If we suppose that $X_n$ tends to~$X$ and that the normalized linear Dirac measures~$(\frac 1 {t_n} \mu_n)_{n\in\N}$ converge to~$\mu$ in the weak-$*$ sense, then we can reformulate the previous result into
	\[\lim_{n\to\infty} \frac 1{t_n} \kTr{K_n} = \Tr X \mu.\] 

	Now if $X$ is ergodic with respect to~$\mu$ then, for almost every~$p$ and for every sequence $t_n\to\infty$, the linear Dirac masses concentrated on the knots $K_n := k_X(p,t_n)$ tend to~$\mu$ in the weak-$*$ sense. Recall that $k_X(p,t_n)$ is the knot obtained by following the orbit of $p$ for a time $t_n$ concatenated with a  geodesic from $\phi_X^{t_n}(p)$ to $p$, as explained in the introduction. Then the convergence implies that $\displaystyle{\lim_{t\to\infty} \frac 1 t \kTr{k_X(p,t)}}$ exists and is equal to~$\Tr X \mu$.
	This proves the second part of Theorem~\ref{T:Continuity}.
\end{proof}



\section{Independence of helicity}
\label{S:Seifert}

As mentioned in the introduction, the helicity is a well-known invariant of vector fields up to volume-preserving diffeomorphism. 
In this section, all vector fields are on the sphere~$\Sph^3$ and preserve a volume form, that we denote by $\Omega$. 
For $X$ such a vector field, Cartan's formula implies that $\iota_X\Omega$ is a closed 2-form, and since the ambient manifold is simply connected it is exact. 
We may then write $\iota_X\Omega=d\alpha$, for $\alpha$ some differential 1-form. 
The helicity of $X$ is defined as
\[\Hel(X)=\int_{\mathbb{S}^3} \alpha\wedge d\alpha,\]
and does not depends on the choice of the primitive $\alpha$.

As we recalled in the introduction most known asymptotic invariants are in fact proportional to a power of helicity~\cite{Arnold, GG, Baader, BM}.
The goal of this section it to prove that the trunkenness of a vector field is not a function of its helicity. 
In order to do so we compute the trunkenness and the helicity of a vector field that preserves the invariant tori of a Hopf fibration of $\Sph^3$.


Considering~$\Sph^3$ as the unit sphere~$\{ (z_1,z_2)\in\C^2, \vert z_1\vert^2+\vert z_2\vert^2 = 1\}$, the \emph{Seifert flow of slope~$(\alpha, \beta)$} is the flow $\phi_{\alpha,\beta}$ given by 
	\[\phi^t_{\alpha,\beta}(z_1,z_2) := (z_1 e^{i2\pi \alpha t}, z_2 e^{i2\pi \beta t}).\]
It is generated by the vector field~$X_{\alpha, \beta}$ given by $X_{\alpha, \beta}(z_1, z_2) = (i2\pi\alpha z_1, i2\pi\beta z_2)$. 
This flow preserves the standard volume form, that is, the volume form $\Omega_\Haar$ associated to the Haar measure of $\Sph^3$.
The flow has two distinctive periodic orbits corresponding to~$z_1=0$ and $z_2=0$ that are trivial knots in~$\Sph^3$.
The tori~$\vert z_1/z_2\vert = r$ for $0<r<\infty$ are invariant and the flow on each of them is the linear flow of slope~$\alpha/\beta$. 
If $\alpha/\beta$ is rational, put $\alpha/\beta=p/q$ with $p,q\in \mathbb{N}$ coprime. 
Then every orbit of~$\phi_{\alpha,\beta}$, different from the two trivial ones, is a torus knot of type~$T(p,q)$.

The helicity of~$\phi_{\alpha,\beta}$ is equal to~$\alpha\beta$. To compute it in the  rational case $(\alpha,\beta)=(p,q)$ with $p,q$ coprime, observe that all the orbits except two are periodic of period~$1$. 
The linking number of an arbitrary pair of such orbits is~$pq$. 
Therefore the asymptotic linking number (also called asymptotic Hopf invariant) equals~$pq$ and, by Arnold's Theorem~\cite{Arnold}, so does the helicity. 
For the general case of $(\alpha, \beta)$ not necessarily rational, it is enough to use the continuity of the helicity, since $X_{\alpha, \beta}$ can be approximated by a sequence of Seifert flows with rational slope. 

\begin{proposition}
	\label{P:Tpq}
	The trunkenness of the Seifert flow~$\phi_{\alpha,\beta}$ with respect to the standard volume form $\Omega_\Haar$ is equal to~$2\min(\alpha,\beta)$.
\end{proposition}

\begin{proof}
Let us first prove $\Tr{X_{\alpha,\beta}} {\Omega_\Haar}\le 2\beta$. 
For this it is enough to exhibit a height function~$h$ that yields $\tr{X_{\alpha,\beta}}{\Omega_\Haar} h = 2\beta$. 
First define $\infty=(0,1)$ and $0=(0,-1)$ in $\Sph^3\subset \mathbb{C}^2$ and take the stereographic projection to identify
$$\{(z_1,z_2) \in \mathbb{C}^2\, ,\, |z_1|^2+|z_2|^2=1\} \simeq \mathbb{R}^3\cup\{\infty\}.$$
Take now as $h$ the standard height function $h_0$ of
$\mathbb{R}^3\cup\{\infty\}$. The spheres are centered at $0\in
\mathbb{R}^3$ that corresponds to the point $(0,-1)\in \Sph^3\subset
\mathbb{C}^2$, hence the orbit $z_1=0$ intersects twice each level
sphere~$h_0^{-1}(t)$. The middle sphere, $S=h_0^{-1}(1/2)$, contains
the other special orbit $z_2=0$ and is the only sphere that intersects
all the orbits of~$\phi_{\alpha,\beta}$. Then  the function~$t\mapsto\int_{h_0^{-1}(t)}\vert\iota_{X_{\alpha,\beta}}\Omega_\Haar\vert$ has a maximum for~$t=1/2$. 

For computing $\int_{S}\vert\iota_{X_{\alpha,\beta}}\Omega_\Haar\vert$, we remark that the 2-sphere~$S$ has the orbit $(e^{i2\pi\alpha t},0)$ as an equator, that the flow is positively transverse to the northern hemisphere and negatively transverse to the southern hemisphere. 
Then the integral $\int_S\vert\iota_{X_{\alpha,\beta}}\Omega_\Haar\vert$ is equal to twice the flux of~$X_{\alpha,\beta}$ through any disc bounded by the curve~$(e^{i2\pi \alpha t},0)$. 
Consider the flat disc~$D$ in~$\Sph^3$ bounded by~$(e^{i2\pi \alpha t},0)$. 
The first return time to~$D$ is constant and equal to~$1/\beta$, so the flux multiplied by~$1/\beta$ gives the total volume of~$\Sph^3$, that is 1. 
Therefore $\Flux{X_{\alpha,\beta}} {\Omega_\Haar} D$ is equal to~$\beta$, and we obtain $\tr{X_{\alpha,\beta}}{\Omega_\Haar} {h_0} = \int_S\vert\iota_{X_{\alpha,\beta}}\Omega_\Haar\vert = 2\beta$. 
By symmetry, we then have~$\Tr{X_{\alpha,\beta}} {\Omega_\Haar} \le 2\min(\alpha,\beta)$.

\bigskip

For proving the converse inequality~$\Tr{X_{\alpha,\beta}} {\Omega_\Haar} \ge 2\min(\alpha, \beta)$, we approximate~$X_{\alpha, \beta}$ by a sequence $(X_{{p_n}/{r_n}, {q_n}/{r_n}})_{n\in\N}$, where $p_n, q_n, r_n$ are integer numbers. 
Theorem~\ref{T:Continuity} yields 
$$\Tr{X_{\alpha, \beta}} {\Omega_\Haar} = \lim_{n\to\infty} \Tr{X_{{p_n}/{r_n}, {q_n}/{r_n}}} {\Omega_\Haar}.$$ 
As the trunkenness is an order-1 invariant (it is multiplied by $\lambda$ if the vector field is multiplied by~$\lambda$), we only have to prove $\Tr{X_{{p}, {q}}} {\Omega_\Haar} = 2\min(p, q)$ for $p,q$ two coprime natural numbers.

Since every orbit of~$X_{p,q}$ is periodic, we can consider a sequence of $(K_n)_{n\in\N}$ of collections of periodic orbits whose induced linear Dirac measures converge to~$\Omega_\Haar$. We take $K_n$ to be an $n$-component link all of whose components are torus knots~$T(p,q)$. 
Actually $K_n$ is a cabling with $n$ strands on~$T(p,q)$, so by Zupan's theorem~\cite{Zupan}, the trunk of $K_n$ is $2n\min(p,q)$. Since the period of each component of~$K_n$  is 1, the total length of~$K_n$ is $n$, and we get $\Tr{X_{p,q}} {\Omega_\Haar} \ge  2\min(p,q)$.
\end{proof}


\begin{proof}[Proof of Theorem~\ref{T:Seifert}]
The previous computations show that for a Seifert flow~$X_{\alpha, \beta}$ on $\Sph^3$ we have $\Hel(X_{\alpha, \beta}, \Omega_\Haar) = \alpha\beta$ and $\Tr{X_{\alpha, \beta}} {\Omega_\Haar}=2\min(\alpha, \beta)$. 
There is no real function $g$ such that $\min(\alpha,\beta)=g(\alpha\beta)$, so there is no function $g$ such that $\Tr{X_{\alpha, \beta}} {\Omega_\Haar} = g(\Hel(X_{\alpha, \beta}, \Omega_\Haar))$.

However the Seifert flows are not ergodic with respect to~$\Omega_\Haar$. 
Indeed, the foliation of~$\Sph^3$ by invariant tori is invariant, so that it is easy to construct an invariant set with arbitrary mesure. 
Still, a theorem of Katok~\cite{Katok} states that every volume-preserving vector field can be perturbed (in the $C^1$-topology) into an ergodic one. 
Starting from $X_{1,8}$ and $X_{2,4}$, we obtain two ergodic volume-preserving vector fields $X'_{1,8}$ and $X'_{2,4}$. 
By continuity, their trunknesses are close to $2$ and $4$ respectively, while their helicities are close to $8$. 
At the expense of multiplying the $X'_{1,8}$ and $X'_{2,4}$ by a constant, we can assume that their helicities are exactly $8$. 
However their trunknesses are still close to $2$ and $4$, hence different.
\end{proof}

The formula~$\Tr{X_{\alpha,\beta}} {\Omega_\Haar} = 2\min(\alpha,\beta)$ is also interesting to compare with Kudryavtseva's and Encisco-Peralata-Salas-Torres de Lizaur's theorems: the function $(\alpha, \beta)\mapsto2\min(\alpha, \beta)$ is continuous but not differentiable, so that trunkenness is a continuous vector field invariant, but it is not \emph{integral regular} in the sense of~\cite{Kud, EPT}.

\section{Trunkenness and the existence of periodic orbits} \label{sec-periodic}

In this section we adress the question of what happens when the infimum in the definition of the trunkenness is a minimum, for  non-singular vector fields on $\Sph^3$ with an invariant mesure $\mu$. 
We deduce that the vector field must posses an unknotted periodic orbit by finding a periodic orbit tangent to a level of the function. 
The proof of Theorem~\ref{T:minperiodic} in particular implies that there is height function $h$ such that 
\[\tr {X} \mu h=\Flux X \mu {h^{-1}(t_{\max})}\] 
for some (not necessarily unique) $t_{\max} \in [0,1]$ and such that all the connected components of $S=h^{-1}(t_{\max})$ along which $X$ is tangent to $S$ are bounded by periodic orbits of $X$. 
The existence of vector fields on~$\Sph^3$ without periodic orbits \cite{Kup94} implies that there are vector fields for which we cannot consider a minimum to define the trunkenness.

\begin{proof}[Proof of Theorem~\ref{T:minperiodic}]
For $f$ a height function we define
\begin{eqnarray*}
	F_f:[0,1] & \to & \mathbb{R}\\
	t & \mapsto &  \Flux X \mu {f^{-1}(t)}
\end{eqnarray*}
Set $h$ the height function such that $\Tr X \mu=\tr {X} \mu h$ and let $t_{\max}\in [0,1]$ be a maximum of $F_h$, hence
\[\Tr X \mu=\tr {X} \mu h=\lim_{\epsilon\to 0} \frac{1}{\epsilon}\,\mu(\phi_X^{[0,\epsilon]}(h^{-1}(t_{\max}))).\]
Denote by  $S$ the level set $h^{-1}(t_{\max})$. Observe that the value $t_{\max}$ is not necessarily unique.

We assume that $\Sph^3$ and  $S$ are oriented and thus we can distinguish three subsets of $S$ (see Figure~\ref{F:TT}):
\begin{itemize}
\item $S^t$ the closed set along wich $X$ is tangent to $S$;
\item $S^+$ the open set along wich $X$ is positively transverse to $S$;
\item $S^-$ the open set along wich $X$ is negatively transverse to $S$.
\end{itemize}
We claim that none of these sets is empty. 
First assume that $S^+$ is empty, then $S=S^t\cup S^-$. 
Observe that $\Sph^3\setminus S$ has two connected components that are diffeomorphic to open 3-dimensional balls. 
Hence there is one of these connected components, denote its closure by $D$, that is invariant under the diffeomorphism of $\Sph^3$ defined by $\phi_X^t$ for any  $t>0$, and since $X$ is non-singular there exists $t_0>0$ such that $\Phi:=\phi_X^{t_0}$ has no fixed points. 
Thus $\Phi$ maps $D$ to~$D$, and $D$ is the closed 3-dimensional disc, hence by Brouwer fixed point theorem $\Phi$ has a fixed point, a contradiction. 
Thus $S^+$ is non-empty, and the same argument proves that $S^-$ is non-empty. 
Now $S^+$ and $S^-$ are open subsets of~$S$ and have empty intersection, hence $S^t=S\setminus(S^+\cup S^-)$ is not empty.

\begin{figure}
	\begin{picture}(60,60)(0,0)
	\put(0,0){\includegraphics[width=.4\textwidth]{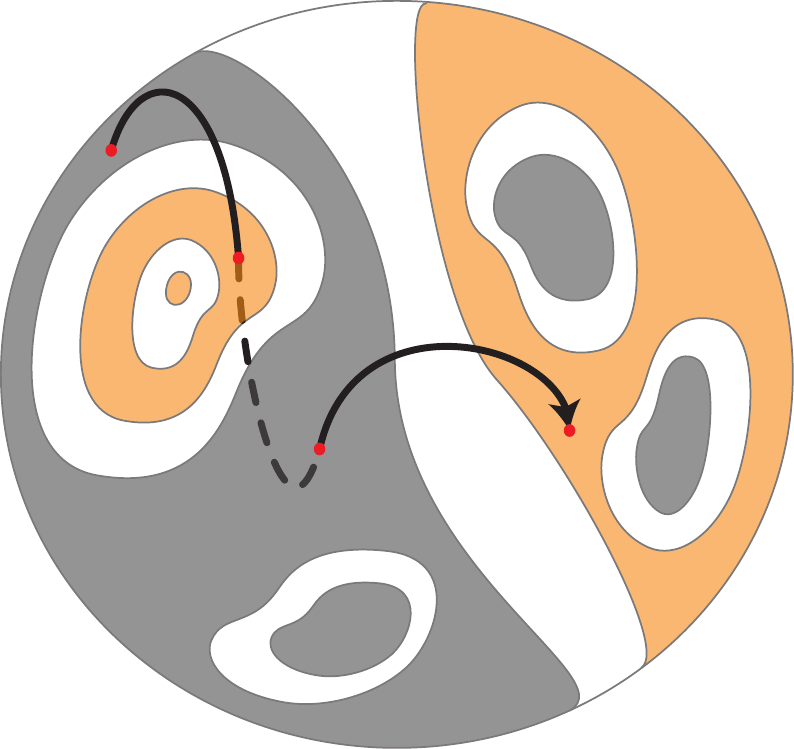}}
	\put(11,12){$S^-$}
	\put(50,40){$S^+$}
	\put(25,50){$S^t$}
	\end{picture}
	\caption{\small The decomposition of the level~$S$ of the height function into $S^+$ where~$X$ is positively transverse to~$S$ (orange), $S^-$ where $X$ is negatively transverse (black), and $S^t$ where $X$ is tangent to~$T$ (white). A piece of orbit of~$X$ that intersects~$S$ four times is also shown.}
	\label{F:TT}
\end{figure}
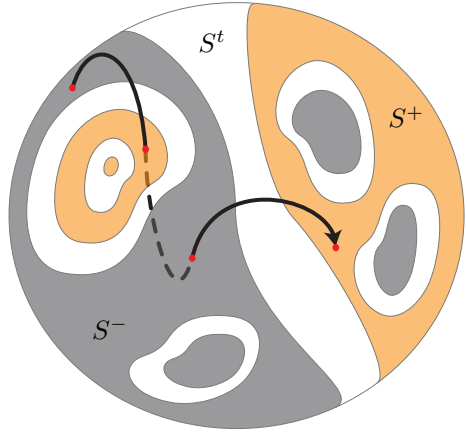

Decompose $B:=\partial S^t$ as $B=B^+\cup B^-$, where $B^+$ and $B^-$ are in the boundary of $S^+$ and $S^-$ respectively. Observe that there is at least one connected component of $S^t$ whose boundary has both positive and negative parts.

Assume that $X$ has no periodic orbits, then $B$ is made of circles and each one of these circles is tranverse to $X$ in at least one point. 

\begin{lemma}\label{L:poincarebendixson}
	If a point $p$ in $S^t$ is such that either its positive or negative orbit is contained in $S$, then $X$ has an unknotted periodic orbit.
\end{lemma}

\begin{proof}
	Assume without loss of generality that the positive orbit of $p$ is contained in $S$, then it limits to some invariant set of $X$ contained in $S$. 
	By Poincar\'e-Bendixson Theorem such a set has to be a periodic orbit of $X$. 
	Since the periodic orbit is contained in the sphere $S$ it has to be unknotted.
\end{proof}

Then under the assumption that $X$ has no unknotted periodic orbits, we have that for every $p\in S^t$ the positive and negative orbit of $p$ have to leave $S$ at some point.
Consider now two points $p,q\in B$ that are in the same orbit and such that the orbit segment connecting them is contained in $S^t$. 
We can thus assume that there exists $\tau\geq 0$ such that
$\phi_X^\tau(p)=q$ and $\phi_X^s(p)\in S^t$ for all $0\leq s\leq \tau$. 
Observe that by allowing $\tau$ to be zero, we consider the case $p=q$.

Take $\epsilon>0$ and consider the orbit segment
\[\mathcal{O}=\{\phi_X^s(p) \mid -\epsilon<s<\tau+\epsilon\}.\]
The flowbox Theorem implies that there is a neighborhood $U$ of $\mathcal{O}$ and a diffeomorphism 
\[\varphi:\mathbb{B}(1)\times (-\epsilon,\tau+\epsilon)\to U\]
so that the flowlines are the image under $\varphi$ of the vertical segments $\{\cdot\}\times (-\epsilon,\tau+\epsilon)$. 
Here $\mathbb{B}(1)$ denotes the 2-dimensional open disc of radius 1.

\begin{proposition}\label{P:Bplusminus}
	There exists $p,q\in B$ as above such that $p\in B^-$ and $q\in B^+$ or $p\in B^+$ and $q\in B^-$.
\end{proposition}

The idea of the proof is that if for every pair of points $p,q$ as above that  are both in $B^+$ or $B^-$, we can change the function $h$ for another height function $h_1$ such that
\[\tr {X} \mu h=\tr X \mu {h_1}\]
and such that the level of $h_1$ realizing the trunkenness has no tangent part that separates $S^+$ from $S^-$, which is impossible.

\begin{proof}
Consider a pair of points $p,q\in B^+$ as above (see Figure~\ref{F:PP}). Let $V_0$ be the neighborhood of $\mathcal{O}$ defined as
$\varphi(\mathbb{D}(1/3)\times(-\epsilon,\tau+\epsilon))$, where  $\mathbb{D}(1/3)$ is the 2-dimensional closed  disc of radius $1/3$ and $V=\varphi(\mathbb{B}(2/3) \times(-\epsilon,\tau+\epsilon))$. Hence $V_0\subset V\subset U$. We will deform the levels of $h$ intersecting $V$ without changing the trunkenness.

\begin{figure}
	\begin{picture}(80,80)(0,0)
	\put(0,0){\includegraphics[width=.5\textwidth]{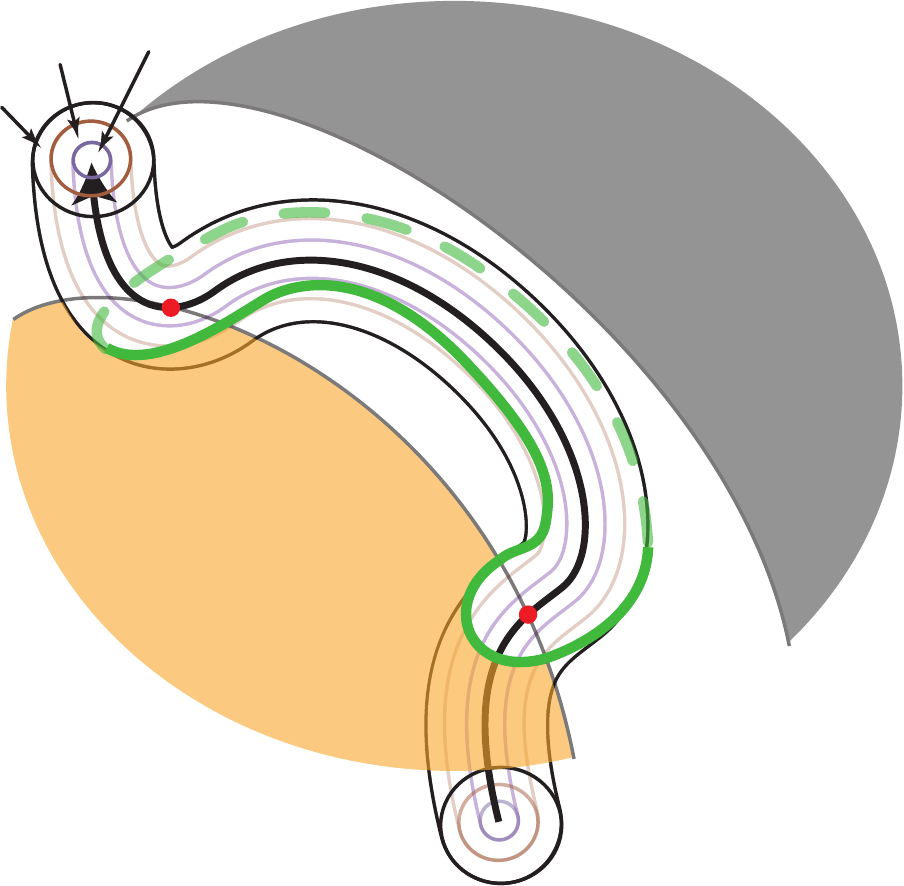}}
	\put(-3,65){$U$}
	\put(3,68){$V$}
	\put(11,70){$V_0$}
	\put(13,28){$S^+$}
	\put(34,38){$S^t$}
	\put(63,48){$S^-$}
	\end{picture}
	\caption{\small A neighborhood~$U$ of a piece of orbit~$\mathcal O$ of~$X$ that connects two points~$p, q$ in~$B^+$. The intersection of~$\partial U$ with $S$ is bold and green.}
	\label{F:PP}
\end{figure}
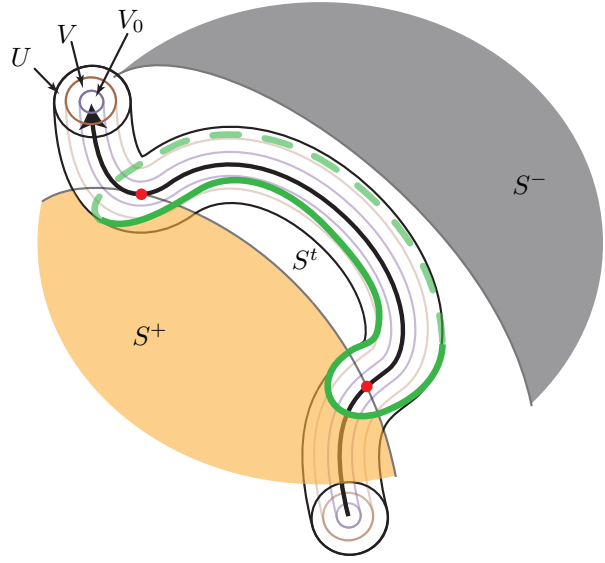

The image of $S$ by $\varphi^{-1}$ defines a surface $\Sigma$ that is
positively transverse to the vertical direction in
$\mathbb{B}(1)\times(-\epsilon,0)$ and
$\mathbb{B}(1)\times(\tau,\tau+\epsilon)$, as in Figure~\ref{F:PP2}
left. If we project $\Sigma$ to $\mathbb{B}(1)\times\{0\}$, the flux
through $\Sigma$ and through the image of $\Sigma$ under the
projection are the same. Thus if we change $\Sigma$ (or any surface) for another surface whose projection, counted with multiplicities and signs, is the same as the one of $\Sigma$, the flux remains constant. 

 Inside $\mathbb{D}(1/3)\times(-\epsilon,\tau+\epsilon)$ we change the images under $\varphi^{-1}$ of the levels of $h$ to obtain a family of surfaces that  are always positively transverse to $X$ and are $C^\infty$-close to the original ones, in such a way that  the projection to $\mathbb{B}(1)\times\{0\}$ is preserved (as in Figure~\ref{F:PP2} right). Using the neighborhood $V$, we can paste the deformed surfaces of $V_0$ with the original surfaces in $U\setminus V$. 
 
\begin{figure}
	\includegraphics[width=.5\textwidth]{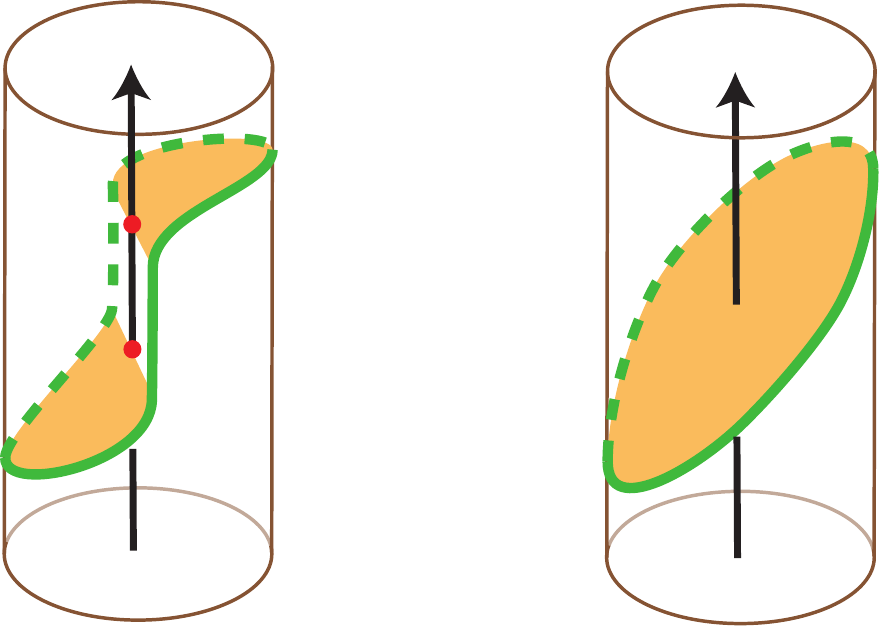}
	\caption{\small Modifying the levels of~$h$ in~$U$ around the arc~$pq$ in the case $p, q\in B^+$.}
	\label{F:PP2}
\end{figure}

 Using $\varphi$ this deformation can be pushed forward to the
 manifold so that the level surfaces of $h$ that intersected $V$ are
 modified and $X$ is positively transverse to the surfaces inside
 $V_0$. Let $\widetilde{S}$ be the surface obtained from $S$ after the
 deformation and $\widetilde{h}$ be a height function whose level sets correspond to the deformed surfaces. Then
\[\tr {X} \mu h=\tr X \mu {\widetilde{h}},\]
since $\Flux X \mu {\widetilde{S}} = \Flux X \mu S$, and for every
surface intersection $V$ the corresponding equation holds.
Clearly, the same proof works if $p,q\in B^-$.

Observe that the level $\widetilde{S}$ of $\widetilde{h}$ realizing the trunkenness coincides with $S$ outside $V$. Moreover, the tangent part $\widetilde{S}^t$ of $\widetilde{S}$ is strictly smaller than $S^t$: the connected component $A$ of $S^t$ containing $p$ and $q$ got  transformed into $A\setminus (V_0\cap A)$. In other words there is now a strip accross $A$.

\bigskip

Assume now that for any couple of points $p,q\in B$ that is joined by an orbit segment contained in $S^t$ we have either $p,q\in B^+$ or $p,q\in B^-$.  The sets $\widetilde{S}^t$ and $B$ are compact, we can thus find a finite number of pairs of points $p,q\in \partial  B$ such that after deforming the levels $h$ as above along each of the corresponding orbit segments, we obtain a height function $h_1$ such that 
\[\tr {X} \mu h=\tr X \mu {h_1}=\Flux X \mu {S_1},\]
where $S_1$ is the level of $h_1$ realizing the trunkenness and such that $S_1^t$ is formed by discs. Hence none of the connected components of $S_1^t$ separates the positively tranverse set $S_1^+$ from the negatively transverse set $S_1^-$, a contradiction.
\end{proof}

Consider now a pair of points $p,q$ as in Proposition~\ref{P:Bplusminus} and assume that $p\in B^-$ and $q\in B^+$. 
We now construct a height function $h_2$ such that 
\[\tr {X} \mu h>\tr X \mu {h_2},\]
a contradiction that implies that $X$ has unknotted periodic orbits.

\begin{figure}
	\includegraphics[width=.5\textwidth]{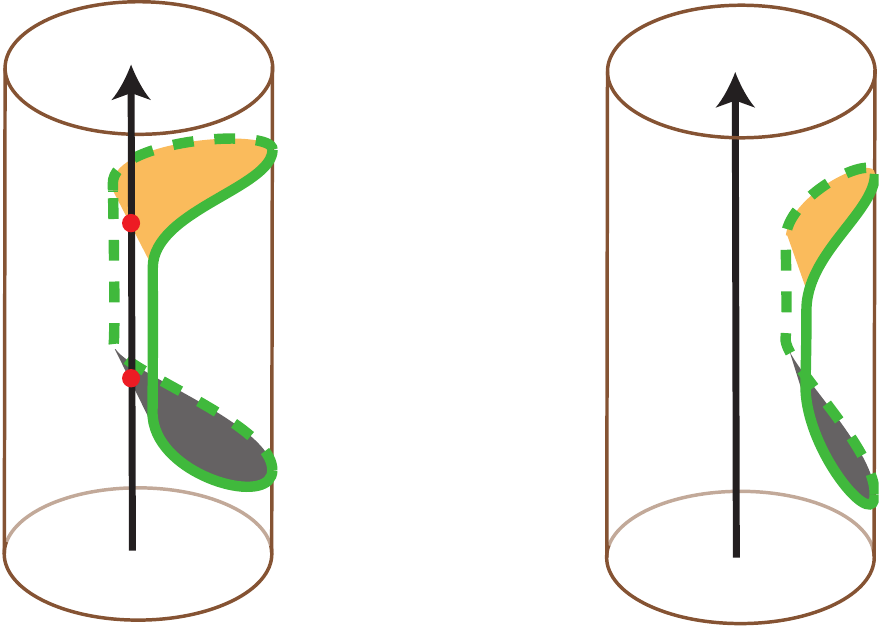}
	\caption{\small Modifying the levels of~$h$ in ~$U$ around the arc~$pq$ in the case $p\in B^+$, $q\in B^-$.}
	\label{F:PM}
\end{figure}

Consider the function $F_h$ and let $M_h$ be the set of $t\in [0,1]$
realizing the trunkenness, that is $M_h=\{t\in[0,1]\mid
F_h(t)=\tr {X} \mu h\}$. Observe that $M_h$ is closed. We distinguish two cases:
\begin{enumerate}
\item all the elements of $M_h$ are isolated;
\item there exist a closed interval $I_{\max}$ contained in $M_h$.
\end{enumerate}

We start with the first case. Take $t_{\max}\in M_h$ and set
$S=h^{-1}(t_{\max})$. To obtain $h_2$ we will deform the levels near
$S$, if $M_h$  has more than one element the following deformation has to be done near each correponding level set.  Thus we assume without loss of generality that $t_{\max}$ is the only element of $M_h$.

Let $\mathcal{O}$, $\varphi$, $U$, $V$ and $V_0$ be defined as above. 
Let $[a,b]\subset [0,1]$ be such that $t_{\max}\in [a,b]$ and if $h^{-1}(t)\cap V_0\neq \emptyset$ then $t\in [a,b]$. 
Shrinking $V$ if necessary, we can assume that $\varphi(\mathbb{B}(2/3)\times (-\epsilon, \tau+\epsilon)$ is contained in $h^{-1}([a,b])$ and that for any $t\in [0,1]\setminus [a,b]$ we have
\[
	\Flux X \mu {h^{-1}(t)}< \tr {X} \mu h-\delta,
\]
for some $\delta>0$. We assume also that $\varphi(\mathbb{B}(2/3)\times [0,\tau])$ is contained in $h^{-1}([a,b])$.

Consider now the surface $\Sigma$ obtained as the image under $\varphi^{-1}$ of $S$ in $U$. Then either $\Sigma$  is positively transverse to the vertical direction in $\mathbb{B}(1)\times(-\epsilon,0)$ and negatively transverse to the vertical direction in $\mathbb{B}(1)\times(\tau,\tau+\epsilon)$, or it is negatively transverse to the vertical direction in $\mathbb{B}(1)\times(-\epsilon,0)$ and positively transverse to the vertical direction in $\mathbb{B}(1)\times(\tau,\tau+\epsilon)$. We assume without loss of generality that $\Sigma$  is negatively transverse to the vertical direction in $\mathbb{B}(1)\times(-\epsilon,0)$ and positively transverse to the vertical direction in $\mathbb{B}(1)\times(\tau,\tau+\epsilon)$, as represented in Figure~\ref{F:PM}.

We want to deform the levels of $h$ intersecting $V_0$ in such a way that their trace under the projection to $\mathbb{B}(1)\times\{0\}$ is reduced. 

Shrinking $V$ if necessary, we can assume that the level sets are:
\begin{itemize}
\item tangent or negatively transverse to $X$ in $\varphi(\mathbb{B}(2/3)\times (-\epsilon,0))$,
\item tangent or positively transverse to $X$ in $\varphi(\mathbb{B}(2/3)\times (\tau, \tau +\epsilon))$.
\end{itemize}
Consider the circle $C=\Sigma \cap \partial(\mathbb{D}(1/3)\times(-\epsilon,\tau+\epsilon))$. 
Since $\partial(\mathbb{D}(1/3)\times(-\epsilon,\tau+\epsilon))$ is topologicaly a sphere, $C$ divides it into two discs. 
Let $W$ be the disc that is entirely tangent to the vertical direction. 
There is a continuous deformation from $\Sigma$ to the continuous surface $\widetilde{\Sigma}$ that coincides with $\Sigma$ outside $\mathbb{D}(1/3)\times (-\epsilon,\tau+\epsilon)$ and coincides with $W$ in this set. 
Observe that the projection of $\widetilde{\Sigma}$ to $\mathbb{B}(1)\times\{0\}$ is a proper subset of the projection of $\Sigma$. 
We can now approximate $\widetilde{\Sigma}$ in $\mathbb{B}(2/3)\times (-\epsilon,\tau+\epsilon)$ with a smooth surface $\Sigma_1$, that can be obtained from $\Sigma$ by a continuous deformation that is the identity near the boundary of $\mathbb{B}(2/3)\times (-\epsilon,\tau+\epsilon)$.

Apply this deformation to nearby levels by pushing (to the right as in Figure~\ref{F:PM}) all the level sets intersecting $\mathbb{B}(1/3)\times (0,\tau)$ with a continuous map that is the identity near the boundary of $\mathbb{B}(2/3)\times (-\epsilon,\tau+\epsilon)$. 
By construction the projection to $\mathbb{B}(1)\times \{0\}$ of the image under $\varphi^{-1}$ of any level set intersecting $V$ is a subset of the original one and for levels sufficiently near $h^{-1}(t_{\max})$ it is a proper subset. 
Thus the flux of $X$ gets reduced on any of the deformed surfaces.

As before, we use $\varphi$ to push forward the deformation to $\Sph^3$ and obtain the levels of a new height function $h_2$ such that $\tr X \mu  {h_2}<\tr {X} \mu h$. 

\bigskip

The proof of the case where $M_h$ has only isolated elements ends here and we are left with the case where the set $M_h$ contains an interval $I_{\max}=[a,b]$. 
Take $t_{\max}=a$ and use the previous procedure to obtain a new height function $h_2$ such that  the interval $I_{\max}$ is reduced to $I_2=[a_2,b]$ for some $a_2>a$. 
This process can be applied as long as the boundary of the tangent part of $h^{-1}(t)$ for any $t\in[a,b]$ is not composed by periodic orbits of $X$. 
Recursively either we find an unknotted periodic orbit or we obtain a height function $f$ such that $\tr X \mu f <\tr {X} \mu h$ as needed.
\end{proof}

The proof of Theorem~\ref{T:minperiodic} gives the following result.

\begin{coro}
Let $X$ be a non-singular vector field on $\Sph^3$, $\mu$ an invariant measure and $h$ a height function  such that 
$$\Tr X \mu=\tr {X} \mu h.$$
Then there is a height function $h_1$ such that:
\begin{itemize}
\item $\tr {X} \mu h=\tr X \mu {h_1}$;
\item for every $s\in[0,1]$ such that $\tr {X} \mu {h_1}=\Flux{X} {\mu} {{h_1}^{-1}(s)}$ and for every point $p$ in the boundary of the tangent part of $S=h_1^{-1}(s)$, either $p$ belongs to a periodic orbit or its orbit limits to a periodic orbit. In both cases the periodic orbit is contained in $S$ and is thus unknotted.
\end{itemize}
\end{coro}

\section{Trunkenness of knotted tubes}
\label{S:Examples}

In this section we compute the trunkenness with respect to a volume form of some vector fields supported in a tubular neighborhood of a link or knot. 
Our statement is reminiscent from Zupan's theorem~\cite{Zupan} concerning the trunk of the cable of a knot. 
Recall that for a divergence-free vector field supported on a tube, the fluxes through all meridian discs are equal.


\begin{proposition}\label{P:KnottedTubes}(see Figure~\ref{F:Tubes} right)
	Suppose that $X$ is a $\Omega$-preserving vector field supported on tubes~$T_1, \dots, T_j$ that are tubular neighborhoods of knots~$k_1,\dots, k_j$, each knot~$k_i$ being non-trivial, and such that $X$ is transverse to all canonical meridian discs and the flux of~$X$ through each of them is constant equal to some~$F\neq 0$, then we have 
	\[\displaystyle{\Tr X \Omega =F\cdot \kTr{k_1\cup\dots\cup k_j}}.\]
\end{proposition}

Note that the trunkenness in this case is independent from the dynamics of~$X$ inside each tube, and in particular of the first-return map on a meridian disc, exactly like the trunk of a cable link is independent of the twist number of the cabling.  
However the restriction that all the knots~$k_i$ are non-trivial is necessary.  
Indeed the trunkenness of a vector field supported on a tube that forms a trivial knot may be smaller than $2F$. 
For example take an unknotted solid torus in $\mathbb{R}^3$ that supports a vector field obtained by the suspension of the identity (with no extra twist) and assume that the vector field is zero outside the solid torus (see Figure~\ref{F:Tubes} top left). 
The trunkenness of such vector field is zero: it is enough to consider a height function whose levels are always tangent to the vector field. 
Observe that the levels of such a height function intersect the solid torus along circles or annuli.
As this example shows, the trunkenness of a vector field supported in the neighborhood of a link with at least one unknotted component seems hard to determine (see Figure~\ref{F:Tubes}).

\begin{figure}
	\includegraphics[width=.7\textwidth]{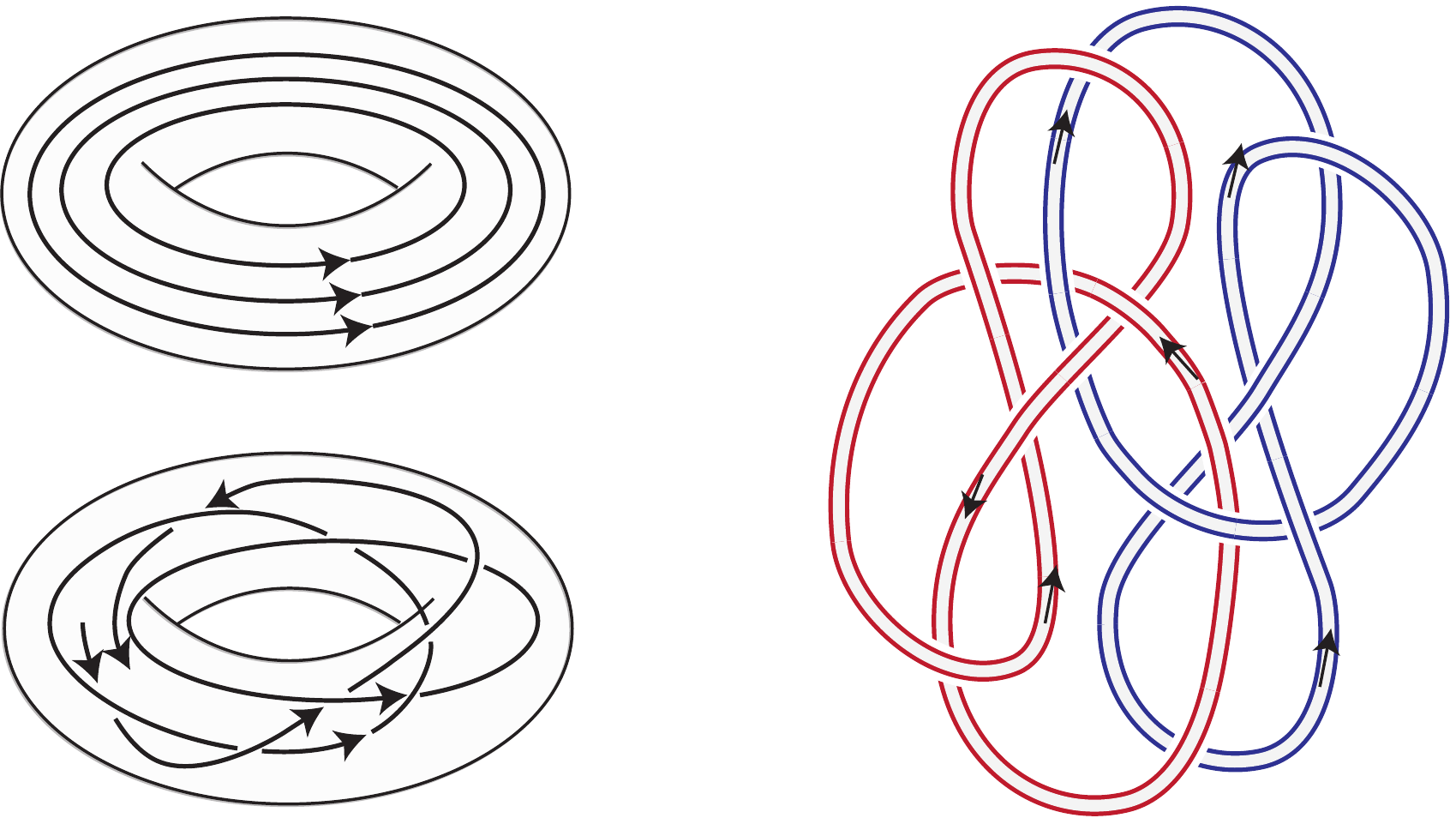}
	\caption{\small On the left, two vector fields supported on the unlink and both transverse to the canonical meridian discs. 
	The top one has all its orbits horizontal, hence its trunkenness is~$0$ while the bottom one has linked orbits. 
	Its trunkenness is at most~$2\cdot\Flux X \Omega D$ but it might be smaller. 
	On the right a vector field supported on a link~$L$ both of which components are figure-eight knots. 
	If the flux of the vector field through the canonical meridian discs of both components is equal to~$F$, then the trunkenness of this vector field is~$F\cdot \kTr L = 8F$ by Proposition~\ref{P:KnottedTubes}.}
	\label{F:Tubes}
\end{figure}

\begin{proof}
Let $L=k_1\cup k_2\cup \ldots\cup k_n\subset \Sph^3$ be a link all of whose components are non-trivial knots. 
Let $h$ be a height function such that $\kTr L =\ktr L h$. 
Observe that we can assume that $h$ exists since for a knot or link the trunk is defined by a minimum (see
Definition~\ref{D:TrunkKnot}). 
The trunk of $L$ is then realized by one or more level sets of $h$, that is, there exist $t_1, t_2,\ldots,t_k$ such that for any $1\leq i\leq k$ there is a critical value $c_i$ of $h|_L$ so that $t_i<c_i<t_{i+1}$ and  $\kTr L=\sharp \{ L\cap h^{-1}(t_i)\}$. 
Let 
$$m_j(t_i)=\sharp \{ k_j\cap h^{-1}(t_i)\}, \qquad \mbox{for} \qquad 1\leq j\leq n \qquad \mbox{and} \qquad 1\leq i\leq k.$$

For any $0<\epsilon<1$ there is an $\Omega$-preserving diffeomorphism
$f_\epsilon$ of $\Sph^3$ that makes the solid tori $T_j$ longer and of
radius $\epsilon r_j$ for any $1\leq j\leq n$, with $r_j$ the radius
of the tubular neighborhood $T_j$. Let $X_\epsilon$ be the
vector field obtained from $X$ via $f_\epsilon$, then
$\Tr X \Omega =\Tr {X_\epsilon} \Omega$ for every $\epsilon$.

Assume that there is a height function $g$ and an index~$i$ such that $\tr X \Omega g< F\cdot \sum_{j=1}^n m_j(t_i)$. 
Consider the height function $g_\epsilon=g \circ f_\epsilon^{-1}$, then Theorem~\ref{T:invariance} implies $\tr X \Omega g=\tr {X_\epsilon} \Omega {g_\epsilon}$ and
$$\tr {X_\epsilon} \Omega {g_\epsilon}<F\cdot\sum_{j=1}^n m_j(t_i).$$
We claim that $\ktr L {g_\epsilon} < \sum_{j=1}^n m_j(t_i)=\kTr L$ holds, which is absurd.

Assume first that $g_\epsilon$ restricted to the core of the tori $f_\epsilon(T_j)$ is a Morse function. 
The core of the tori form the link $L$ thus $\ktr L {g_\epsilon}\geq \kTr L $ in one hand. 
But since
\[
	\tr {X_\epsilon} \Omega {g_\epsilon}
	=\Flux {X_\epsilon} \Omega {g_\epsilon^{-1}(T)}
	<F\cdot\min_{1\leq i\leq k} \sum_{j=1}^n m_j(t_i)
\]
for some $T\in[0,1]$, taking $\epsilon$ arbitrarily small tells us that the number of discs in the intersection of $g_\epsilon^{-1}(T)$ and the tori is smaller than $\kTr L$. 
In other words we conclude that $\ktr L {g_\epsilon}<\kTr L$, which is impossible.
We are left with the case where $g_\epsilon$ is not a Morse function when restricted to~$L$, seen as the core of the tori $f_\epsilon(T_j)$. 
The following lemma finishes the proof of the proposition.

\begin{lemma}
	In the previous context, there exists another height function $g_\epsilon''$ so that:
	\begin{itemize}
		\item $g_\epsilon''|_L$ is a Morse function;
		\item $\tr{X_\epsilon} \Omega {g_\epsilon}= \tr {X_\epsilon} \Omega {g_\epsilon''}$.
	\end{itemize}
\end{lemma}

\begin{proof}
Let $S=g_\epsilon^{-1}(t)$ be a level that is tangent to at least one of the components of $L$ and let $k_1$ be one of the tangent components. 
We will modify $g_\epsilon$ in a neighborhood of $k_1$, this modification taking place near the part of~$k_1$ that is tangent to $S$.

Observe that $S$ is tangent to a closed subinterval of $k_1$, since if it were tangent to all of $k_1$ then $k_1$ would  be the trivial knot, in contrary to our assumption on $L$. 
For $\delta<\epsilon r_1$, consider $N(\delta,k_1)$ the $\delta$-tubular neighborhood of $k_1$ that is contained in $f_\epsilon(T_1)$. 
For $\delta$ sufficiently small, the intersection $S\cap N(\delta,k_1)$ is composed by a finite number of discs. 
Let $D_1$ be one of these discs, then $\partial D_1$ is included in $\partial N(\delta,k_1)$ and is either a contractible circle in $\partial N(\delta,k_1)$ or not.

If $\partial D_1$ is contractible in~$\partial N(\delta,k_1)$, then we can locally deform the levels of $g_\epsilon$
near $D_1$, to obtain a new height function $g_\epsilon'$ such that:
\begin{itemize}
	\item the disc $D_1'$ obtained from $D_1$ is tangent to $k_1$ at one point;
	\item near $D_1'$ all levels are transverse to $k_1$;
	\item $\tr {X_\epsilon} \Omega {g_\epsilon}=\tr {X_\epsilon} \Omega {g_\epsilon'}$.
\end{itemize}
The support of the deformation is contained in $N(\delta,k_1)$.
Performing this deformation near any disc in $S\cap N(\delta,k_1)$ whose boundary is contractible in $\partial N(\delta,k_1)$, we can assume now that the discs of $S$ that are tangent to $k_1$ have non-contractible boundary in $\partial N(\delta,k_1)$. 
Thus these discs are homologous to meriodional discs $\mathbb{D}^2\times \{\theta\}$ of $T_1$. 
Again, we can locally deform the levels, so that this discs are transverse to $k_1$. 
The deformed height function $g_\epsilon''$ is Morse when restricted to~$L$. 

Thus $g_\epsilon''$ is Morse when restricted to $k_1$. 
Repeating this process if necessary, we obtain a that $g_\epsilon''$ is Morse when restricted to $L$. 
\end{proof}
\end{proof}


\bibliographystyle{siam}

\end{document}